\documentclass[11pt, reqno]{amsart}

\usepackage{latexsym}
\usepackage{amssymb}
\usepackage{mathrsfs}
\usepackage{amsmath}
\usepackage{fancybox,color}
\usepackage{enumerate}
\usepackage[latin1]{inputenc}

\newcommand{\refeq}[1]{(\ref{#1})}

 \def\1{\raisebox{2pt}{\rm{$\chi$}}}

\newtheorem{theorem}{Theorem}[section]
\newtheorem{corollary}[theorem]{Corollary}
\newtheorem{lemma}[theorem]{Lemma}
\newtheorem{proposition}[theorem]{Proposition}
\newtheorem{remark}[theorem]{Remark}

\newcommand{\R}{{\mathbb R}}
\newcommand{\RR}{{\mathbb R}}

\newcommand{\N}{{\mathbb N}}

\newcommand{\E}{{\mathbb E\,}}

 \newcommand{\eps}{{\varepsilon}}
 \def\1{\raisebox{2pt}{\rm{$\chi$}}}
 
\newcommand{\abs}[1]{\left|#1\right|}

\newcommand{\Rn}{\mathbb{R}^n}
\newcommand{\osc}{\operatorname{osc}}

\def\vint_#1{\mathchoice%
          {\mathop{\kern 0.2em\vrule width 0.6em height 0.69678ex depth -0.58065ex
                  \kern -0.8em \intop}\nolimits_{\kern -0.4em#1}}%
          {\mathop{\kern 0.1em\vrule width 0.5em height 0.69678ex depth -0.60387ex
                  \kern -0.6em \intop}\nolimits_{#1}}%
          {\mathop{\kern 0.1em\vrule width 0.5em height 0.69678ex
              depth -0.60387ex
                  \kern -0.6em \intop}\nolimits_{#1}}%
          {\mathop{\kern 0.1em\vrule width 0.5em height 0.69678ex depth -0.60387ex
                  \kern -0.6em \intop}\nolimits_{#1}}}
\def\vintslides_#1{\mathchoice%
          {\mathop{\kern 0.1em\vrule width 0.5em height 0.697ex depth -0.581ex
                  \kern -0.6em \intop}\nolimits_{\kern -0.4em#1}}%
          {\mathop{\kern 0.1em\vrule width 0.3em height 0.697ex depth -0.604ex
                  \kern -0.4em \intop}\nolimits_{#1}}%
          {\mathop{\kern 0.1em\vrule width 0.3em height 0.697ex depth -0.604ex
                  \kern -0.4em \intop}\nolimits_{#1}}%
          {\mathop{\kern 0.1em\vrule width 0.3em height 0.697ex depth -0.604ex
                  \kern -0.4em \intop}\nolimits_{#1}}}

\newcommand{\kint}{\vint}

\newcommand{\aveint}[2]{\mathchoice%
          {\mathop{\kern 0.2em\vrule width 0.6em height 0.69678ex depth -0.58065ex
                  \kern -0.8em \intop}\nolimits_{\kern -0.45em#1}^{#2}}%
          {\mathop{\kern 0.1em\vrule width 0.5em height 0.69678ex depth -0.60387ex
                  \kern -0.6em \intop}\nolimits_{#1}^{#2}}%
          {\mathop{\kern 0.1em\vrule width 0.5em height 0.69678ex depth -0.60387ex
                  \kern -0.6em \intop}\nolimits_{#1}^{#2}}%
          {\mathop{\kern 0.1em\vrule width 0.5em height 0.69678ex depth -0.60387ex
                  \kern -0.6em \intop}\nolimits_{#1}^{#2}}}
\newcommand{\ud}{\, d}
\newcommand{\half}{{\frac{1}{2}}}

\newcommand{\ol }{\overline}
\newcommand{\Om}{\Omega}
\newcommand{\I}{\textrm{I}}
\newcommand{\II}{\textrm{II}}
\newcommand{\dist}{\operatorname{dist}}
\renewcommand{\P}{\mathbb{P\,}}
\newcommand{\F}{\mathcal{F}}
\newcommand{\om}{\omega}
\newcommand{\trm}{\textrm}

\begin{document}

\title[p-harmonic Harnack via stochastic games]{Harnack's inequality for p-harmonic functions via stochastic games}

\author[Luiro]{Hannes Luiro}
\address{Department of Mathematics and Statistics, University of
Jyv\"askyl\"a, PO~Box~35, FI-40014 Jyv\"askyl\"a, Finland}
\email{hannes.s.luiro@jyu.fi}

\author[Parviainen]{Mikko Parviainen}
\address{Department of Mathematics and Statistics, University of
Jyv\"askyl\"a, PO~Box~35, FI-40014 Jyv\"askyl\"a, Finland}
\email{mikko.j.parviainen@jyu.fi}

\author[Saksman]{Eero Saksman}
\address{Department of Mathematics and Statistics, University of
Helsinki, PO~Box~68, FI-00014 Helsinki, Finland}
\email{eero.saksman@helsinki.fi}

\thanks{The third  author was
supported by the Finnish CoE in Analysis and Dynamics Research.}

\date{\today}

\begin{abstract} We give a  proof of Lipschitz continuity of p-harmonious functions, that are tug-of-war game analogies of ordinary p-harmonic functions. This result is used to obtain  a new proof of Harnack's inequality for p-harmonic functions in the case $p>2$ that  avoids classical techniques like Moser iteration, but instead relies on suitable choices of strategies for the stochastic tug-of-war game.

\end{abstract}

\maketitle

\section{Introduction}

Considerable progress was made in the mid-1950s and -l960s in the regularity theory
of elliptic equations due to the discoveries of DeGiorgi \cite{degiorgi57}, Nash \cite{nash58} and Moser \cite{moser60,moser61}.
DeGiorgi and Nash proved that solutions to certain elliptic partial differential equations are H\"older continuous whereas Moser showed
that non-negative solutions satisfy the Harnack inequality. Such an inequality can be
used, in turn, to prove the H\"older continuity of solutions. This development settled Hilbert's 19th problem.

In this paper, we present a completely new and rather straightforward proof for 
the Harnacks's inequality for solutions to $p$-Laplace equation
\[
\begin{split}
{\Delta_p} u=\operatorname{div}(\abs{\nabla u}^{p-2} \nabla u)=0
\end{split}
\]
for $2< p<\infty$. It is worth emphasizing that, quite surprisingly, our method readily  produces Lipschitz continuity estimates.
The proof utilizes a recently discovered connection between stochastic games and $p$-harmonic functions, see \cite{peress08}.   The argument is based on a choice of a strategy, and is thus completely different from the proofs of De Giorgi, Moser, or Nash.

Fix $\eps>0$ and consider the following two-player zero-sum-game from \cite{manfredipr11}. At the beginning, a token is placed at a point $x_0\in
\Om\subset \Rn$ and the players toss a biased coin with probabilities
$\alpha$ and $\beta$, $\alpha + \beta =1$. If they get
heads (probability $\alpha$), they play a tug-of-war, that is, a fair coin is tossed and the winner
of the toss is allowed to move the game position to any
$x_1\in B_\eps (x_0)$. On the other hand, if they get tails
(probability $\beta$), the game state moves according to  uniform
probability (with respect to normalized Lebesgue measure) to a random point in the ball $B_\eps(x_0)$.
Then they continue playing the same game from  $x_1$.
Once the game position reaches the boundary, Player II pays Player I
the amount given by a pay-off function $F$. Naturally, Player I tries to maximize and Player II tries to minimize the expected payoff, and in this way we obtain the value of the game, which is denoted by $u_\varepsilon$. The value of this game approximates a $p$-harmonic function when we choose suitable $\alpha$ and $\beta$ according to $p$ and $n$. The pay-off function is defined in a suitable $\varepsilon$-neighbourhood of the boundary, and it can be thought as the prescribed boundary values of $u_\varepsilon .$ In the limit $\varepsilon\to 0$ the value function $u_\varepsilon$ tends to a $p$-harmonic function with, loosely speaking, boundary values $F.$

We obtain Lipschitz continuity (Theorem~\ref{thm:lipschitz-continuity}), an oscillation estimate (Corollary~\ref{strong-oscillation-estimate}), and Harnack's inequality (Theorem~\ref{thm:pharmoniousharnack}) for game values. Similarly, we obtain an oscillation estimate and Harnack's inequality for $p$-harmonic functions (respectively Corollary~\ref{p-oscillation-estimate} and end of Section~\ref{sec:p-harmonic-harnack}).
 In the case $p>n$, a player can force the game to a point with a positive probability and the proof for Harnack's inequality is even  simpler (Corollary~\ref{cor:harnack-p>n}).

The argument for Theorem~\ref{thm:lipschitz-continuity} is based on a cancellation strategy, in which Player II has two goals: he tries to to cancel the effect of Player I and move to a target direction fixed number of times. If Player II reaches his goals, which occurs with high enough probability, then the expected value has a symmetric component. Finally, we can utilize this symmetry in comparing the values for the games starting from different points.

The classical linear interplay between harmonic functions
and martingales is well known.
In the nonlinear case, a connection of the tug-of-war game to so called infinity harmonic functions was established by Peres, Schramm, Sheffield and Wilson in \cite{peresssw09}, see also \cite{legruyera98}, \cite{oberman05}, and \cite{legruyer07}. This has inspired further studies to many different directions, see for example \cite{peresps10}, \cite{atarb10}, \cite{manfredipr10d} as well as led to simplified  proofs in the theory of PDEs, see for example \cite{armstrongs10}.

\section{Preliminaries}

Let us start by fixing  the basic notation used throughout the work. We denote
\[
B_{{\rho}}(x_0)=\{\,x\in \Rn\, :\, |x-x_0|<{\rho}\,\}.
\]
When no confusion arises, we drop the common center point in statements and denote
 $B_r$, $B_{2r}$ etc.
 The integral average of $u$ is denoted
by
\[
\kint_{B_{\rho}} u(x) \ud
x=\frac{1}{\abs{B_{\rho}}}\int_{B_{\rho}} u(x) \ud x,
\]
where
$\abs{{B_{\rho}}}$ denotes the Lebesgue measure of ${B_{\rho}}$.

Fix $p>2.$ Let  $\Omega\subset\mathbb{R}^{n}$ be a bounded domain and fix $\eps >0$. To prescribe boundary values we introduce
 the compact boundary strip  of width $\eps$  by setting
\[
\begin{split}\Gamma_\eps :=
\{x\in \RR^n \setminus \Om\,:\,\dist(x,\partial \Om) \leq \eps\}.
\end{split}
\]
We next recall  the two-player zero-sum-game called 'tug-of-war with noise'. At the beginning, a token is placed at a point $x_0\in
\Om$ and the players toss a biased coin with probabilities
$\alpha$ and $\beta$, $\alpha + \beta =1$.  Here
\begin{equation}\label{eq:constants}
\alpha=\alpha(p):= \frac{p-2}{n+p}\quad \mbox{and}\quad \beta=\beta(p):=\frac{n+2}{n+p}.
\end{equation}
If they get
heads (probability $\alpha$), they play a tug-of-war, that is, a fair coin is tossed and the winner
of the toss is allowed to move the game position to any
$x_1\in B_\eps (x_0)$. On the other hand, if they get tails
(probability $\beta$), the game state moves according to the uniform
probability to a random point in the ball $B_\eps(x_0)$.
Then they continue playing the same game from  $x_1$.

This procedure yields a  sequence of game states
$x_0,x_1,\ldots$ where  every $x_k$ is a random
variable. We denote by $x_\tau \in \Gamma_\eps$ the first point in
$\Gamma_\eps$ in the sequence, where  $\tau$ refers to the first
time we hit $\Gamma_{\eps}$. The payoff is $F(x_\tau)$, where
$F:\Gamma_\eps
\to \R$ is a given, bounded, Borel measurable
\emph{payoff function}. Player I earns $F(x_\tau)$ while
Player II earns $-F(x_\tau)$.

A history  of a game up to step $k$ is a vector of the first $k+1$
game states $x_0,\ldots,x_k$ and $k$ coin tosses $c_1,\ldots,c_k$, that is,
\begin{equation}
\label{eq:history}
\begin{split}
\left(x_0,(c_1,x_1),\ldots,(c_k,x_{k})\right).
\end{split}
\end{equation}
Here $c_j\in \mathcal C:=\{0,1,2\},$ where $0$ denotes that Player I wins, $1$ that Player II wins and $2$ that a random step occurs.

A strategy
$S_\I$
for Player I is a collection of  Borel-measurable  functions that give the next
game position given the history of the game i.e.\ next move as a function
of all previously played moves and all previous coin tosses. For example
\[
S_\I{\left(x_0,(c_1,x_1),\ldots,(c_k,x_k)\right)}=x_{k+1}\in  B_\eps(x_k)
\]
if Player I wins the toss. Similarly Player II  plays according to a strategy $S_{\II}$.

Let $\Om_\eps=\Om\cup\Gamma_\eps\subset \Rn$. The space of all game sequences (and our probability space) will be
\[
\begin{split}
H^\infty= x_0\times \left({\mathcal C},\Om_\eps\right)\times\ldots.
\end{split}
\]
Writing $\om=\left(x_0,(c_1,x_1),\ldots\right)\in H^\infty$, define the random variable  time $\tau$ by 
\begin{equation}
\label{eq:stopping-time-boundary}
\begin{split}
\tau(\om)=\inf\{k\,:\,x_k\in \Gamma_\eps,k=0,1,\ldots\}.
\end{split}
\end{equation}
This $\tau(\om)$ is a stopping time relative to the filtration $\{\F_k\}_{k=0}^\infty$, where
${\mathcal F}_0:=\sigma (x_0)$ and
\begin{equation}\label{eq:filtration}
{\mathcal F}_k:=\sigma (x_0, (c_1,x_1),\ldots, (c_k,x_k))\quad \mbox{for}\quad k\geq 1.
\end{equation}

The fixed starting point $x_{0}$ and the strategies $S_{\I}$ and
$S_{\II}$ determine a unique probability measure
$\mathbb{P}^{x_0}_{S_I,S_\II}$ on the natural product $\sigma$-algebra. In particular, this measure is defined on the sets of the type $x_0\times \left(C_1, B_1\right)\times\ldots$, where $C_i\subset \mathcal C$ and the $B_i\subset \Om_\eps$  are Borel subsets.
The probabilility measure
is  built by
applying  Kolmogorov's extension theorem to  the family of
transition probabilities  (compare to \cite[Section 2]{manfredipr11})
 \begin{equation}
\label{eq:meas-steps}
\begin{split}
&\pi_{S_\I,S_\II}\left(x_0,(c_1,x_1)\ldots,(c_k,x_k),(C,A)\right)\\
&= \frac{\alpha}{2}
\delta_0(C)\delta_{S_\I\left(x_0,(c_1,x_1)\ldots,(c_k,x_k)\right)}(A)
+\frac{\alpha}{2}\delta_1(C)\delta_{S_\II\left(x_0,(c_1,x_1)\ldots,(c_k,x_k)\right)}(A)\\
&\hspace{1 em}+\beta \delta_{2}(C)
\frac{ \abs{{A}\cap B_\eps(x_k)}}{\abs{B_\eps(x_k)}},
\end{split}
\end{equation}
as long as $x_k\in\Omega$, otherwise if  $x_k\not\in \Omega $, the transition probability forces $x_{k+1}=x_k.$

The expected payoff, when starting from $x_0$ and using the strategies $S_\I,S_\II$, is
\begin{equation}
\label{eq:defi-expectation}
\begin{split}
\mathbb{E}_{S_{\I},S_{\II}}^{x_0}[F(x_\tau)]
&=\int_{H^\infty} F(x_\tau(\om)) \ud \mathbb{P}^{x_0}_{S_\I,S_\II}(\om).
\end{split}
\end{equation}

Note that, due to the fact that $\beta>0$, or equivalently
$p<\infty$, the game ends almost surely
$$\mathbb{P}^{x_0}_{S_\I,S_\II}(\{\omega\in H^{\infty}\colon \tau(\omega) < \infty\})=1$$
for any choice of strategies. Namely, since $\Omega$ is bounded we may choose $N_0\geq 1$ large enough so that $N_0\varepsilon >2\, {\rm diam}(\Omega ).$ Then the random walk  with step uniformly distributed in $ {B}_\varepsilon (0)$ jumps out of $\Omega$ with a  positive probability, uniform with respect to the initial  point.  The  claim follows by observing that almost surely the game will contain infinitely many blocks  of length $N_0$  consisting of solely  random moves.

 The \emph{value of the game for Player I}
is given by
\[
u^\eps_\I(x_0)=\sup_{S_{\I}}\inf_{S_{\II}}\,\mathbb{E}_{S_{\I},S_{\II}}^{x_0}[F(x_\tau)]
\]
while the \emph{value of the game for Player II} is given by
\[
u^\eps_\II(x_0)=\inf_{S_{\II}}\sup_{S_{\I}}\,\mathbb{E}_{S_{\I},S_{\II}}^{x_0}[F(x_\tau)].
\]
For basic properties  of the value functions (like measurability issues) we refer to \cite{MS}.

Observe that history contains more information than in \cite{manfredipr10d,manfredipr11} where the history only contained the previous game positions and the strategies were defined accordingly, for example $ S_\I(x_0,x_1,\ldots,x_k)=x_{k+1}\in  B_{\eps}(x_{k})$.
Nevertheless, intuitively the player can not do better than to step into a maximum/minimum of the underlying value function, and thus this formalism produces the same value functions as we shall verify next.
 \begin{proposition}
\label{thm:u-v-same}
 Let $\Om\subset \Rn$ be a bounded open set. Let $u$ be a value function for Player I defined with history \eqref{eq:history}, and $v$ the corresponding value function defined with history that only contains the previous game positions. Both functions have the same boundary values $F$. Then
 \[
 \begin{split}
 u=v.
 \end{split}
 \]
 The same result also holds for the value function for Player II.
\end{proposition}

\begin{proof}
We show that by choosing a strategy according to the minimal
values of $v$, and by adding an arbitrarily small correction term, Player II can make the process a supermartingale.
The optional stopping theorem then implies that $u$ is bounded by $v$ up to a small correction. The reverse direction is obtained by a similar argument

Player I follows any strategy and Player II follows a strategy
$S_\II^0$ such that at $x_{k-1}\in \Om$ he chooses to step to a
point that almost minimizes $v$, that is, to a point $x_k \in B_\eps (x_{k-1})$ such that
\[
v(x_k)\leq\inf_{ B_\eps(x_{k-1})} v+\eta 2^{-k}
\]
for some fixed $\eta>0$ (when proving the reverse,  player $I$, in turn,  tries to almost maximize the value of $v$).
We start from the point $x_0$. It follows that
\[
\begin{split}
&\mathbb{E}_{S_\I, S^0_\II}^{x_0}[v(x_k)+\eta 2^{-k}\,|\F_{k-1}]
\\
&\leq \frac{\alpha}{2} \left\{\inf_{ B_\eps(x_{k-1})} v+\eta 2^{-k}+\sup_{  B_\eps(x_{k-1})}
v\right\}+ \beta \kint_{ B_{\eps}(x_{k-1})}
v \ud y+\eta 2^{-k}
\\&\leq v(x_{k-1})+\eta 2^{-(k-1)},
\end{split}
\]
where we have estimated the strategy of Player I by $\sup$.
We also used the fact that according to \cite{manfredipr10d}, $v$ satisfies the dynamic programming principle
\[
\begin{split}
v(x_{k-1})= \frac{\alpha}{2} \left\{\inf_{ B_\eps(x_{k-1})} v+\sup_{  B_\eps(x_{k-1})}
v\right\}+ \beta \kint_{ B_{\eps}(x_{k-1})}
v \ud y.
\end{split}
\]
Thus
\[
M_k=v(x_k)+\eta 2^{-k}
\] is a supermartingale with respect to the filtration $\{ {\mathcal F_k}\}_{k\geq 0}$
defined in \refeq{eq:filtration}.

Also later on  our martingale considerations are always with respect to this filtration. We
deduce
\[
\begin{split}
u(x_0)&= \sup_{S_{\I}}\inf_{S_{\II}}\,\mathbb{E}_{S_{\I},S_{\II}}^{x_0}[F(x_\tau)]\le \sup_{S_\I} \mathbb{E}_{S_\I,
S^0_\II}^{x_0}[F(x_\tau)+\eta 2^{-\tau}]\\
&\leq\sup_{S_\I}  \mathbb{E}^{x_0}_{S_\I, S^0_\II}[M_0]=v(x_0)+\eta,
\end{split}
\]
where we used
 the optional stopping theorem on the uniformly  bounded martingale $M_{k}$.
 Since $\eta$ was arbitrary this proves the claim.
\end{proof}

Now, the results in \cite{manfredipr11} are directly at our disposal. Especially, by combining the previous Proposition with \cite[Theorems 1.2 and 1.4]{manfredipr11} we obtain
\begin{proposition}
\label{lem:quoted-results}
It holds that $u_\eps:=u_\I^\eps=u_\II^\eps$, the value  $u_\eps$ is the unique Borel measurable function with fixed boundary values $F$ that satisfies the dynamic programming principle, i.e.
\begin{equation}
\label{eq:dpp}
\begin{split}
u_\eps(x)= \frac{\alpha}{2} \left\{\inf_{ B_\eps(x)} u_\eps+\sup_{  B_\eps(x)}
u_\eps\right\}+ \beta \kint_{ B_{\eps}(x)}
u_\eps \ud y,
\end{split}
\end{equation}
for $x\in \Om$. In addition, $u_\eps$ can alternatively be defined using \eqref{eq:dpp}, and is sometimes called $p$-harmonious function.
\end{proposition}

Concrete choices of the strategy for one of the players can be used to estimate $p$-harmonious functions via the following observation

\begin{lemma}
\label{lem:value-estimate}
Let $\tau^*$ be a stopping time with respect to the filtration \refeq{eq:filtration} with $\tau^*\leq\tau $, where $\tau$ is the stopping time in \eqref{eq:stopping-time-boundary}. Then
\[
\begin{split}
u_{\eps}(y)\leq\sup_{S_{\I}}\mathbb{E}^y_{S_{\I},S^0_{\II}}(u_{\eps}(x_{\tau^*})),
\end{split}
\]

for any fixed $S^0_\I$. Similarly
\[
\begin{split}
u_{\eps}(y)\geq\inf_{S_{\II}}\mathbb{E}^y_{S^0_{\I},S_{\II}}(u_{\eps}(x_{\tau^*})),
\end{split}
\]
for any fixed $S^0_\II$.
\end{lemma}

\begin{proof}
By symmetry it  is enought to prove the first statement. Assume that we are given a fixed strategy $S^0_\II$. Let then
Player I  follow  again (compare the proof of Proposition \ref{thm:u-v-same}) a strategy
$S^\trm{max}_\I$ such that at $x_{k-1}\in \Om$ he chooses to step to a
point that almost maximizes $u_\eps$, that is, to a point $x_k \in B_\eps (x_{k-1})$ such that
\[
u_\eps(x_k)\ge \sup_{ B_\eps(x_{k-1})} u_\eps-\eta 2^{-k}
\]
for some fixed $\eta>0$. By this choice and \eqref{eq:dpp}, we get
\[
\begin{split}
&\mathbb{E}_{S^\trm{max}_\I, S^0_\II}^{x_0}[u_\eps(x_k)-\eta 2^{-k}\,|\mathcal F_{k-1}]
\\
&\ge \frac{\alpha}{2} \left\{\inf_{ B_\eps(x_{k-1})} u_\eps+\sup_{  B_\eps(x_{k-1})}
u_\eps-\eta 2^{-k}\right\}+ \beta \kint_{ B_{\eps}(x_{k-1})}
u_\eps \ud y-\eta 2^{-k}
\\&\ge  u_\eps(x_{k-1})-\eta 2^{-(k-1)}.
\end{split}
\]
Hence under these strategies $M_k=u_\eps(x_k)-\eta 2^{-k}$ is a submartingale.
Again by the optional stopping theorem  it follows that
\[
\begin{split}
\sup_{S_{\I}}\mathbb{E}_{S_{\I},S^0_{\II}}^{y}[u_\eps(x_\tau^*)]&\ge  \mathbb{E}_{S^\trm{max}_\I, S^0_\II}^{y}[u_\eps(x_\tau^*)]\\
&\ge \mathbb{E}_{S^\trm{max}_\I, S^0_\II}^{y}[u_\eps(x_\tau^*)-\eta 2^{-\tau^*}]\\
&\ge \mathbb{E}^y_{S^\trm{max}_\I, S^0_\II}[M_0]=u_\eps(y)-\eta.
\end{split}
\]
Since $\eta>0$ is arbitrary, this completes the proof.
\end{proof}

\section{Local Lipschitz estimate}

In this section, we show that the value function for the game is asymptotically Lipschitz continuous. By passing to limit with the step size, we obtain a new and direct game theoretic proof for Lipschitz continuity of $p$-harmonic functions with $p>2$. In the next section we then obtain  the Harnack inequality as a quick corollary.  The proof of the Lipschitz regularity (and Harnack's inequality) is based on the choice of strategies, and is thus completely different from the proofs based on the works of De Giorgi, Moser, or Nash, see \cite{degiorgi57}, \cite{moser60,moser60}, and \cite{nash58}.

The  proof  of the Lipschitz estimate is divided in two parts. In the first one (the 'linear' part of the proof) we estimate the probability of a specialized random walk (called 'cylinder walk') to hit a certain part of a boundary first. The needed estimate is  fairly standard,  and we give a complete proof for the readers convenience in the appendix. The second part (the 'non-linear part') contains  the  core of the argument, and its proof is quite transparent as it applies a naturally  chosen strategy for the tug-of-war game.

\subsection*{Cylinder walk}
Constants $\alpha,\beta >0$ with $\alpha+\beta =1$ are determined by the exponent  $p>2$ as before in \refeq{eq:constants}.
Consider the following random walk (called the `cylinder walk') in  a $n+1$ -dimensional cylinder. Suppose that we are at a point $(x_j,t_j)\in B_{2r}(0)\times[0,2r]$, where $r>0$ is fixed.  With probability $\alpha/2$ we move to the point $(x_j,t_{j}-\eps)$, and with $\alpha/2$ to $(x_j,t_{j}+\eps)$. With probability $\beta$ we move to the point $(x_{j+1},t_j)$, where $x_{j+1}$ is randomly chosen from the ball $B_{\eps}(x_j)$. 

The next lemma gives a quite intuitive estimate for the probability that the cylinder walk escapes though the bottom; the proof is given in the appendix.
\begin{lemma}
\label{lemma:speed}
 Let us start the cylinder-walk  from the point $(0,t)$. Then the probability that the walk does not escape the cylinder through its bottom is less than
$$
C(p,n)(t+\varepsilon)/r,
$$
for all $\eps>0$ small enough.
\end{lemma}

\medskip

We are ready for one of our main results.
\begin{theorem}
\label{thm:lipschitz-continuity}
Let $u_{\eps}$ be a $p$-harmonious positive function in $\Om$ and assume that $B_{10r}(z_0)\subset\Om$ and $r>\varepsilon .$ Then
\[|u_{\eps}(x)-u_{\eps}(y)|\leq C(p,n)r^{-1}|x-y|(\sup_{B_{6r}(z_0)} u_{\eps}-\inf_{B_{6r}(z_0)} u_{\eps}))\]
for all  $x,y\in B_{r}(z_0)$ with $|x-y|\geq\eps$.
\end{theorem}
\begin{proof}
 Let $x,y\in B_r(z_0)$ and choose integer $m$ comparable to $\abs{x-y}/\eps$, for example $(m-1) \eps\le |x-y|< m\eps$. Fix $z\in B_{2r}(z_0)$ so that $|z-x|=|z-y|=(m-1) \eps$.    We define a strategy $S^0_{\II}$ for Player II for the game that starts from $x$. The simple idea is the following: he always tries to cancel the earliest move of Player I which  he has not yet been able to cancel. If all the moves at that moment are cancelled and he wins the coin toss, then he moves vector $$\eps\left(\frac{m-1}{m}\right)\frac{z-x}{|z-x|},$$where the factor $1-\frac1{m}=\frac{m-1}{m}$ is due to the fact that the players cannot step to the boundary of $B_\eps(x)$.

At every moment we can divide the game position as a sum of vectors
\[
x+\sum_{k\in I_1}u^1_k+\sum_{k\in I_2}u_k^2+\sum_{k\in I_3}v_k\,.
\]
Here $I_1$ denotes the indices of rounds when Player I has moved, vectors $u^1_k$ are her moves, and correspondingly the $u_k^2$ represent the moves of Player II.
Set $I_3$ denotes the indices when we have taken a random move, and these vectors are denoted by $v_k$.

Then we define a stopping time ${{\tau^*}}$ for this process. We give three conditions to stop the game:
\begin{description}
 \item[(i)]
If Player II has won $m$ more times than Player I.
\item[(ii)]
If Player I has won at least $\frac{2r}{\eps}$ times more than Player II.
\item[(iii)]
If $|\sum_{k\in I_3}v_k|\geq 2r$.
\end{description}
This stopping time is finite with probability $1$, and does not depend on strategies nor starting points.

Let us then consider the situation when the game has ended by reason (i). In this case,
the sum of the steps made by players  can easily  be computed, and it is
\[
\begin{split}
\sum_{k\in I_1}u^1_k+\sum_{k\in I_2}u_k^2=m\left(\frac{m-1}{m}\right)\frac{\eps(z-x)}{|z-x|}=z-x\,.
\end{split}
\]
Thus the point $x_{{\tau^*}}$ is actually randomly chosen (radially weighted) around $z$. To be more precise, in case (i), we have
\begin{equation}
\label{eq:xtau}
\begin{split}
x_{{\tau^*}}=z+\sum_{k\in I_3}u_{k}\,.
\end{split}
\end{equation}
On the other hand, in the cases (i) and (ii)
there is actually never more than $\frac{2r}{\eps}+m$
vectors which are not cancelled in the sum
\[
\sum_{k\in I_1}u^1_k+\sum_{k\in I_2}u_k^2.
\]
This, combined with condition (iii) guarantees that when the game is running, we never exit $B_{6r}(z_0)$.

The most crucial point of our proof is the following cancellation effect. Let $S^0_I$ be the corresponding cancellation strategy for Player I when starting from $y$. Then
\begin{equation}
\label{eq:symmetry}
\begin{split}
\mathbb{E}^x_{S_{\I},S^0_{\II}}&[u_{\eps}(x_{{\tau^*}})\,|\trm{ game ends by condition (i)}]\\
=&\mathbb{E}^{y}_{S^0_{\I},S_{\II}}[u_{\eps}(x_{{\tau^*}})\,|\trm{ game ends by condition (i)}]
\end{split}
\end{equation}
for any choice of the strategies $S_{\I}$ or $S_{\II}$. Indeed, if the game ends by the condition (i), then $x_{{\tau^*}}$ in \eqref{eq:xtau} has only the random part which is independent on the strategies and the point where the game started.
By Lemma~\ref{lem:value-estimate}, it holds that
\[
\begin{split}
u_{\eps}(x)\leq\sup_{S_{\I}}\mathbb{E}^x_{S_{\I},S^0_{\II}}[u_{\eps}(x_{{\tau^*}})].
\end{split}
\]
Similarly,
\[
\begin{split}
u_{\eps}(y)\geq\inf_{S_{\II}}\mathbb{E}^y_{S^0_{\I},S_{\II}}[u_{\eps}(x_{{\tau^*}})].
\end{split}
\]
 Let us denote by $P$ the probability that the game ends by the condition (i), which only depends on $p$ and $n$. Using the above estimates, the equation  \eqref{eq:symmetry} to eliminate the symmetric part, and recalling that  when the game is running, we never exit $B_{6r}(z_0)$, we get
\[
\begin{split}
|u_{\eps}(x)-u_{\eps}(y)|&\le\abs{\sup_{S_{\I}}\mathbb{E}^x_{S_{\I},S^0_{\II}}[u_{\eps}(x_{{\tau^*}})]
-\inf_{S_{\II}}\mathbb{E}^y_{S^0_{\I},S_{\II}}[u_{\eps}(x_{{\tau^*}})]}\\
&\le \Big|P\big(\mathbb{E}^x_{S_{\I},S^0_{\II}}[u_{\eps}(x_{{\tau^*}})\,|\,\trm{(i)}]-\mathbb{E}^y_{S^0_{\I},S_{\II}}[u_{\eps}(x_{{\tau^*}})\,|\,\trm{(i)}]\big)\Big|\\
 &\hspace{1 em}+(1-P)(\sup_{B_{6r}(z_0)}u_{\eps}-\inf_{B_{6r}(z_0)}u_{\eps})\\
&\le (1-P)(\sup_{B_{6r}(z_0)}u_{\eps}-\inf_{B_{6r}(z_0)}u_{\eps}).
\end{split}
\]

Finally, we verify that $1-P$ is small enough. For that end observe that with the choice $t=m\varepsilon$ the  estimate for the `cylinder walk' in Lemma \ref{lemma:speed} gives the upper  bound for $1-P$ in the form
$$
1-P\leq  C(p,n)(m+1)\varepsilon/r\leq C'|x-y|/r
$$
for $|x-y|\geq \varepsilon .$
\end{proof}

By a standard reasoning the previous result holds also for  $x,y\in 6B,$ and  if $|x-y|<\varepsilon,$ one may choose $z$ with $|x-z|=|y-z|=\varepsilon$ and estimate $|u_\eps(x)-u_\eps(x)|\leq|u_\eps(x)-u_\eps(z)|+|u_\eps(z)-u_\eps(y)|$. As a direct corollary of the previous result we hence obtain
\begin{corollary}
\label{strong-oscillation-estimate}
Let $u_{\eps}$ be a $p$-harmonious positive function in $\Om$. Then 
\[
\osc(u_\eps,B_\rho)\leq C(n,p)\frac{\rho}{r} \osc(u_\eps,B_{r})
\]
whenever $\eps\leq \rho\leq r$ and $B_{2r}\subset \Om$.
\end{corollary}

Let us recall that by  \cite[Theorem 1.6]{manfredipr11} we may obtain  $p$-harmonic functions $u$ (in the viscosity sense) as locally uniform limits of $p$-harmonious functions $u_\varepsilon$, whose  boundary values coincide with $u.$   Moreover, viscosity solutions to $p$-Laplacian coincide with the classical ones by \cite{juutinenlm01} (see \cite{juju11} for easier proof of this fact). Hence we may let $\varepsilon\to 0$ in the previous result to obtain

\begin{corollary}
\label{p-oscillation-estimate}
Let $u$ be a $p$-harmonic function in $\Om$. Then 
\[
\osc(u,B_\rho)\leq C(n,p)\frac{\rho}{r} \osc(u,B_{r})
\]
whenever $0<\rho<r$ with $B_{2r}\subset \Om$.

\end{corollary}

\section{Proof of Harnack's inequality}
\label{sec:p-harmonic-harnack}

We first establish  a lemma that deduces a Harnack type bound for a given function $u$  just assuming a H\"older type oscillation estimate and a fairly weak assumption on the possible blowup of the function $u$ in small balls. Using this the $p$-harmonic Harnack will be a simple consequence of the Lipschitz bounds in the previous section and a simple comparision  argument with the fundamental solution.

Lemma \ref{lemma:hannesargument} below deduces Harnack from just
Hölder type oscillation estimate together with a rather mild growth condition (see \refeq{eq:oski2}), which in particular holds for $p$-harmonic functions, as is verified after the proof of the lemma. Hence it  can also be used to prove Harnack if the oscillation estimate is obtained by other means, like using the De Giorgi method.

\begin{lemma}
\label{lemma:hannesargument}
 Let $u$ be a positive and continuous function in $B_5(0)\subset\R^n$ normalized by $u(0)=1$  and that satisfies for some constant $C\geq 1$ and  exponent $\gamma >0$ the oscillation estimate
\begin{equation}\label{eq:oski1}\
{\rm osc}\, (u,B_r(x))\leq C\big(\frac{r}{R}\big)^\gamma \sup_{B_R(x)}u
\end{equation}
for $|x|\leq 2$ and $0<r<R\leq 1.$
Assume also that for some exponent $\lambda >0$ one has
\begin{equation}\label{eq:oski2}
\inf_{B_r(x)}u\leq Cr^{-\lambda}, 
\end{equation}
for $|x|\leq 2$ and $r\leq 1.$
Then
$$
\sup_{B_1(0)}u \leq c(n,C,\gamma,\lambda ):=\big( 2^{1+\lambda}C)^{1+\lambda/\gamma},
$$
where $C$ is larger of the constants in \eqref{eq:oski1} and \eqref{eq:oski2}.
\end{lemma}
\begin{proof} Set $R_k= 2^{1-k}$ for $k\in \mathbb N.$ Let  $x_1=0$ and select $x_2$
from $\{ |x|=1\}$ so that  $$u(x_2)=\max_{\overline{B}_1(0)}u=\max_{\overline{B}_{R_1}(x_1)}u.$$ Continue inductively by choosing $x_{k+1}$ from $\ol {B}_{R_k}(x_k)$ so that
$$
u(x_{k+1})=\max_{\overline{B}_{R_{k}}(x_k)}u=:M_k,\quad \mbox{for}\quad k\in \mathbb N.
$$
The idea of the proof  is to observe that in a relatively small neighbourhood of $x_k$
condition \refeq{eq:oski2} shows that $u$ has to take a relatively small value, and when this value is compared to $M_{k-1}=u(x_k)$, the oscillation estimate forces $u$ to take relatively large value in point $x_{k+1}$ near to $x_k$. This allows us to deduce that $M_{k}/M_{k-1}$
is bounded from below by some constant larger than 1.  By iterating this we reach a contradiction. All this will work assuming that $M_1$ is large enough.

To commence with the detailed proof, set  $\delta:=(2^{1+\lambda}C)^{-1/\gamma}$ and assume  contrary to the claim that
\begin{equation}\label{eq:M1}
M_1\geq \big( 2^{1+\lambda}C)^{1+\lambda/\gamma}.
\end{equation}
We claim that this implies
\begin{equation}\label{eq:tieto}
M_k\geq 2C(\delta R_{k+1})^{-\lambda}\quad {\rm for}\:\: k\geq 1.
\end{equation}
Since $R_k\to 0$ as $k\to\infty$ and  $|x_k|\leq R_1+R_2\ldots =2$ for all $k$,  this shows that $u$ is unbounded in $\overline{B}_2(0)$.  Hence (\ref{eq:M1}) is impossible and the statement of the lemma  follows.

It remains to verify (\ref{eq:tieto})  assuming  (\ref{eq:M1}). For $k=1$  one checks  inequality (\ref{eq:tieto})  directly from the choice of $\delta$ and (\ref{eq:M1}). Assume that it is true for all indices
 $k\leq j$, where $j\geq 1.$ Pick any $k\in\{ 2,\ldots , j+1\}$. By (\ref{eq:oski2})  and  the induction hypothesis on (\ref{eq:tieto}) we have
$$
\inf_{B_{\delta R_k}(x_{k})}u\leq C(\delta R_k)^{-\lambda}\leq M_{k-1}/2=u(x_k)/2.
$$
 An application of  the oscillation condition  to the concentric discs $B_{\delta R_{k}}(x_k)\subset B_{ R_{k}}(x_k)$  together with the previous estimate yields that
\begin{eqnarray}\label{eq:kasvu1}
M_{k}&=&\sup_{B_{R_k}(x_k)}u\geq C^{-1}\delta^{-\gamma}\big(\sup_{B_{\delta R_k}(x_k)} u-\inf_{B_{\delta R_k}(x_k)}u\big)\nonumber\\
&\geq& C^{-1}\delta^{-\gamma}\big(u(x_k)-\frac{u(x_k)}{2}\big)\nonumber \\
&=& C^{-1}\delta^{-\gamma}\frac{M_k}{2}.
\end{eqnarray}
By writing this for all $k\in\{ 2,\ldots j+1\}$ and multiplying the equations together we obtain
\begin{equation}\label{eq:kasvu12}
M_{j+1}\geq \big(C^{-1}\delta^{-\gamma}/2\big)^jM_1.
\end{equation}
We see that  condition (\ref{eq:tieto}) holds for $k=j+1$ if one has
\begin{equation}\label{eq:kasvu2}
\big(C^{-1}\delta^{-\gamma}/2\big)^jM_1\geq 2C(\delta R_{j+2})^{-\lambda},
\end{equation}
and a simplification shows that  due to (\ref{eq:M1}) this holds for the chosen value of $\delta$.
\end{proof}

\medskip

\begin{proof}[Proof of Harnacks inequality for p-harmonic functions] By scaling and translation it is enough to check that a positive $p$-harmonic function $u$ in the domain $B_5(0)$, normalized with $u(0)=1$ satisfies the conditions  of the previous lemma.

The oscillation estimate (\ref{eq:oski1})  with $\gamma=1$ was achieved in Corollary \ref{p-oscillation-estimate}. In turn,  (\ref{eq:oski2}) is obtained with simple comparison with the fundamental solution. We may assume that $0\not\in \overline{B}_r(z).$ Set $\kappa(p):=(n-p)/(p-1)$. Consider first the case $p<n$ and the fundamental solution in $B_3(z)\setminus \overline{B}_r(z)$
$$
v(x):= 2\frac{(|x-z|^{-\kappa(p)}-3^{-\kappa(p)})}{(|z|^{-\kappa(p)}-3^{-\kappa(p)})}.
$$
If $u_{|\partial B_r(z)}\geq v_{|\partial B_r(z)}$, then $v\leq u$ in $B_3(z)\setminus \overline{B}_r(z)$ and hence $u(0)\geq v(0)=2,$ which is a contradiction. Thus, since $u$ attains the infimum at the boundary of $B_r(z)$, we have
$$
\inf_{B_r(z)}u\leq 2\frac{(r^{-\kappa(p)}-3^{-\kappa(p)})}{(|z|^{-\kappa(p)}-3^{-\kappa(p)})} \leq c(p)r^{-\kappa (p)},
$$
whence one may choose $\lambda =\kappa (p)$ in (\ref{eq:oski2}).

If $p=n$, similar comparison with the judiciously chosen fundamental solution $c_1\log (1/|x-z|) +c_2$ produces even stronger (logarithmic) estimate, and we may choose $\lambda >0$ arbitrarily in (\ref{eq:oski2}). Finally, the same conclusion is obtained in  case $p>n$ by employing the fundamental solution $c_1|x|^{-\kappa (p)} +c_2$, where now $\kappa (p)<0$.
\end{proof}

\begin{remark}
{\rm
It is interesting to note that in the given proof of  lemma  \ref{lemma:hannesargument} one cannot much weaken the H\"older type  of the assumed  oscillation bound.}
\end{remark}

\section{Harnack's inequality for $p$-harmonious functions}

The proof of Harnack's inequality for  $p$-harmonious functions essentially uses the idea of lemma \ref{lemma:hannesargument}. Minor changes are needed due to the fact that the Lipschitz estimate of Theorem \ref{thm:lipschitz-continuity} is valid only for distances of order $\gtrsim\varepsilon .$  This difficulty  will be overcome by stopping the iterative construction in the proof of lemma \ref{lemma:hannesargument}
at the level $R_k\sim \varepsilon$ and invoking the second part of the following lemma. In turn, the first  part of the lemma verifies that $p$-harmonious functions obey
 the condition (\ref{eq:oski2}) up to level $r\sim \varepsilon$.

\begin{lemma}\label{lemma:apuja}
 Assume that $u_\varepsilon$ is a positive $p$-harmonius function  $(p>2$ and $\varepsilon <1/10)$ in $B_5(0)\subset\R^n$, normalized by
$u_\varepsilon(0)=1$. Then

\noindent {\rm (i)}\quad $u_\varepsilon$ satisfies  condition {\rm \refeq{eq:oski2}} of Lemma {\rm\ref{lemma:hannesargument}} with $\lambda =n$,  for $r\geq\varepsilon$, i.e
\begin{equation}\label{eq:oski3}
\inf_{B_r(z)}u_\eps\leq Cr^{-n}, \nonumber
\end{equation}
for ${|z|\leq 2}$ and  $r\in (2\varepsilon, 1)$, where $C$ depends only on $p$ and $n$.

\noindent  {\rm (ii)}\quad If $x,y\in B_3(0)$ with $|x-y|\leq 10\varepsilon,$ then $$u_\varepsilon (x)\geq c u_\varepsilon (y),$$ where $c$  only depends on $p$ and $n$.

\end{lemma}

\begin{proof}
Let $U=B_4(z)\setminus \overline{B}_r(z)$, where we may assume that $0\in U.$ We consider
  the game that starts from $x_0=0$ and in which Player I  uses the strategy $S^0_\I$, where she always moves a maximal step  towards $z$.
The game is stopped at the $\varepsilon$-boundary $\Gamma_\varepsilon $ of $ U$  and employing the  boundary values $(u_\eps)_{|\Gamma_\eps}$.
The corresponding stopping time is denoted by ${\tau^*}$.

In order to estimate the probability of stopping at the inner boundary, we use an auxiliary function $v$ that is
 harmonic in $U$ with boundary values
$$
\left\{ \begin{array}{ll}
 v=1 &\ {\rm on} \;\; \overline{B}_r(z) ,\\
v=0 &\ {\rm on} \;\; \partial  B_4(z),\\
\end{array}\right.
$$
In other words $v(x)=(|x-z|^{2-n}-4^{2-n})(r^{2-n}-4^{2-n})^{-1}$ if $n\geq 3$ and $v(x)=\log(4/|x-z|)(\log 4/r)^{-1}$ in case $n=2$. Actually, we employ $v$ in $\Gamma_\eps \cup U$, but the above formulas hold verbatim.

Let us observe that
in both the cases there is a constant $c>0$ that  depends only on dimension so that
\begin{equation}\label{eq:v0}
v(0)\geq cr^n.
\end{equation}

Then $v(x_k)$ is a submartingale  for any strategy $S_{\II}$ for player II since
\begin{equation}
\label{eq:submartingale}
\begin{split}
&\mathbb{E}_{S^0_\I, S_\II}^{x_0}[v(x_{k+1})|\,x_0,x_1,\ldots,x_{k}]\\
&\geq \beta \kint_{B_\eps (x_{k})} v \ud y +
\frac{\alpha}{2}\Big(
v\big(x_k+\varepsilon\frac{ (z-x_k)}{|z-x_k|}\big)+v\big(x_k-\varepsilon\frac{ (z-x_k)}{|z-x_k|}\big)\Big)\\
&\geq \beta v(x_k)+\alpha v(x_k)=v(x_k),
\end{split}\nonumber
\end{equation}
where  we used the mean value property of $u$ together with the facts that $y\mapsto v(x+y)$ is radially decreasing and  the map
$0< t\mapsto v(x+ty_0)$ is convex for any fixed $y_0\not=0.$
Denote by $P$ the probability of stopping at the inner boundary. The optional stopping theorem  yields  (independent of the strategy $S_{\II}$)
\begin{equation}
cr^n\leq v(0)= v(x_0)\leq {\mathbb{E}}^{x_0}_{S^0_\I,S_\II}[ v(x_{\tau*})]\leq 2^{n+1}
P,\nonumber
\end{equation}
since $v\le 2^{n+1}$ in $B_{r-\eps}(z)\setminus  B_r(z)$ and $v\le 0$ in $B_{4+\eps}(z)\setminus  B_4(z)$. Thus $P\geq c'r^n.$
Finally, we obtain  by Lemma \ref{lem:value-estimate}
$$
1=u_\varepsilon (x_0)\geq \inf_{S_\II}{\mathbb{E}}^{x_0}_{S^0_\I,S_{\II}}[u_\varepsilon (x_{\tau^*})]
 \geq P\inf_{B_r(z)}u_\varepsilon\geq c'r^n\inf_{B_r(z)}u_\varepsilon\ ,
$$
and the claim follows.

In order to prove (ii), consider simply the game that starts from $x$, where Player $I$ uses the strategy $S^0_\I$ where he takes $(\varepsilon/2)$-step  towards $y$, and actually jumps to $y$ if this is in his reach. The game is stopped at $\tau^{**}$ as we reach either $y$ or the $\varepsilon$-boundary $\overline B_{2+\varepsilon}(x)\setminus  B_{2}(x).$ Obviously the probability to stop at $y$ exceeds $(\alpha/2)^{20},$ whence again by Lemma  \ref{lem:value-estimate} we obtain that
$$
u_\eps(x)=u_\varepsilon (x_0)\geq \inf_{S_\II}{\mathbb{E}}^{x_0}_{S^0_\I,S_{\II}}[u_\varepsilon (x_{\tau^{**}})]\ge (\alpha/2)^{20} u_\eps(y),
$$
which proves (ii) with $c=(\alpha/2)^{20}.$
\end{proof}

We are ready to establish Harnack's inequality for $p$-harmonious functions. The idea of the proof follows  that of Lemma \ref{lemma:hannesargument},  the necessary
modifications are due to two facts: first of all, $u_\varepsilon$ needs not to be continuous whence the choice of the points $x_n$ must be adjusted. Secondly, Corollary \ref{strong-oscillation-estimate} is valid only in the range $r>\varepsilon $ so that iteration has to be stopped before that level.
\begin{theorem}\label{thm:pharmoniousharnack} Let $\varepsilon<\varepsilon_0(p,n)$ and
assume that $u_\varepsilon$ is  a positive $p$-harmonious function  in $B_{10}(0)$, normalized by $u_\varepsilon (0)=1.$ Then
$$\sup_{B_1(0)}u_\varepsilon \leq C(p,n).$$
\end{theorem}

\begin{proof} Corollary \ref{strong-oscillation-estimate} yields that
\begin{equation}\label{eq:oski3}
{\rm osc}\, (u_\eps,B_r(x))\leq C\big(\frac{r}{R}\big) {\rm osc}\, (u_\eps,B_R(x))
\end{equation}
for $|x|\leq 4$ and $\varepsilon<r<R\leq 2.$
We may increase $C$ if needed so that it also works for Lemma \ref{lemma:apuja}.
Due to this lemma we also have the counterpart of \refeq{eq:oski2} for $r>\varepsilon,$ with $\lambda=n.$ Actually, it also holds with $\lambda=2n,$ since in this case the statement is weaker, and we can run the
proof of Lemma \ref{lemma:hannesargument} with this exponent. In particular, we
choose $\delta:= (2^{2n+1}C)^{-1}$and use the counter assumption \refeq{eq:M1}.

We  only  indicate the changes, assuming that the reader keeps  the proof of Lemma \ref{lemma:hannesargument} in mind: this time we use the iteration to   select  $x_{k+1}$ from $\overline{B}_{R_k+\varepsilon}(x_k)\setminus B_{R_k}(x_k)\subset \overline{B}_{2R_k}(x_k)$ by the maximum principle for $p$-harmonious functions so that
$$M
_k:=u_\eps(x_{k+1})\geq \sup_{\overline{B}_{R_{k}}(x_k)}u_\eps.
$$
This will be performed  for $k=1,\ldots ,k_0$, where $k_0$ is chosen so that $\delta R_{k_0}\in (2\varepsilon, 4\varepsilon].$
We observe that $|x_{k+1}-x_k|\leq 2R_k,$ and hence all the points $x_k$
lie inside the ball $B_4(0).$

As before we obtain \refeq{eq:kasvu12}  for $j=k_0-1.$ Together with Lemma \ref{lemma:apuja} (i)  it  implies that
\begin{equation}
\begin{split}
c\; \geq\; &\frac{\sup_{x\in B_{\delta R_{k_0}}(x_{k_0})}u_\eps(x)}{\inf_{ x\in B_{\delta R_k}(x_{k_0})}u_\eps(x)}
\geq \frac{u_\eps(x_{k_0})}{C(\delta R_{k_0})^{-n}}=\frac{M_{k_0-1}}{C(\delta R_{k_0})^{-n}}\nonumber\\
&\geq \frac{(C^{-1}\delta^{-1}/2)^{k_0-2}M_1}{C(\delta 2^{1-k_0})^{-n}}
\geq C''(C\delta 2^{n+1})^{-k_0}.\nonumber
\end{split}
\end{equation}
Above the first inequality comes from Lemma \ref{lemma:apuja} (ii) and the observation that diam$(B_{\delta R_{k_0}})\leq 8\varepsilon .$
We have $\delta= (2^{2n+1}C)^{-1},$ so that $C\delta 2^{n+1}<1$
and the above inequality yields a contradiction is $k_0$ is large enough, i.e. if
$\varepsilon >0$ is small enough.
\end{proof}

\begin{remark}\label{remark:applications}{\rm
One of the main results of \cite{manfredipr11} is the convergence  $u_\varepsilon\to u$, where $u$ is the p-harmonic function with the same boundary values as the functions  $(u_\varepsilon)_{\varepsilon >0}$, see \cite[Theorem 1.6]{manfredipr11} for the precise formulation. In that result the boundary values are assumed to be continuous  in order first to obtain of the existence a uniform continuous limit in \cite[Corollary 4.7]{manfredipr11}. The $p$-harmonicity of the limit is then established through viscosity theory. One may apply our Theorem \ref{thm:pharmoniousharnack} (alternatively already the Lipschitz estimate Corollary \ref{strong-oscillation-estimate}) to obtain a quick  argument for \cite[Corollary 4.7]{manfredipr11}), and generalize it to the case where one just assumes the boundedness from the boundary values.
}\end{remark}

\section{Harnack's inequality when $p>n$}
\label{sec:harnack-p>=n}

There is an alternative proof for Harnack's inequality when $p>n$. This proof is based on a stronger fact that in this case a player has a strategy of forcing the game position to a point with a uniform positive probability  before exiting a larger ball.
This result is based on the use of the fundamental solution to the $p$-Laplace equation and an iteration argument, see Proposition 1.1 (ii) in  \cite{peress08}. The proof can be modified for the version of the game considered in this paper, and thus we have

\begin{theorem}
\label{thm:p>n} Let $p>n$, $\Om=B_1(0)\setminus \{0\}$ and $x_0$ a starting point such that $\abs{x_0}<\half$. Then Player I has a strategy such that the probability of reaching $\{0\}$ before exiting $B_1(0)$ is uniformly larger than zero for all small $\eps$.
\end{theorem}

\begin{corollary}[Harnack]
\label{cor:harnack-p>n}
Let $p>n$ and $u_\eps$ be a positive $p$-harmonious function in $\Om$. Then
\[
\begin{split}
\sup_{B_\rho(z_0)} u_\eps \le C \inf_{B_\rho(z_0)} u_\eps
\end{split}
\]
for $B_{5\rho}(z_0)\subset \Om$ with a constant $C$ independent of $u$.
\end{corollary}
\begin{proof}
Let $\eta>0$, and choose a point $x\in B_\rho(z_0)$ so that
$$u_\eps(x)\le
\inf_{B_\rho(z_0)} u_\eps+\eta,$$
 and $z\in B_\rho(z_0)$ so that
 $$u(z)\ge \sup_{B_\rho(z_0)} u_\eps -\eta.$$
  Denote $r=\abs{x-z}<2\rho$. According to
Theorem~\ref{thm:p>n}, when at $x$, Player I has a strategy $S_\I^0$ such that the probability $P$ of reaching $z$ before exiting $B_{2r}(z)$ is uniformly larger than zero. We define a stopping time
\[
\begin{split}
\tau^*:=\inf\{k\,:\,x_k\in \{z\}\cup (\Rn\setminus B_{2r}(z))\}.
\end{split}
\]
By the above choices and Lemma~\ref{lem:value-estimate}, we have
\[
\begin{split}
\inf_{B_\rho(z_0)} u_\eps+\eta&\ge u_\eps(x)
\ge\inf_{S_{\II}}\mathbb{E}^{x}_{S^0_{\I},S_{\II}}(u_{\eps}(x_{\tau^*}))\ge P u_\eps(z)+(1-P)\cdot 0\\
&\ge P( \sup_{B_\rho(z_0)} u_\eps -\eta).\qedhere
\end{split}
\]
\end{proof}

It is worth noting that in this range the PDE proofs  also become easier, see Serrin~\cite{serrin64} as well as \cite{koskelamv96,manfredi94} in the case $p>n-1$. There is also an unpublished manuscript by T. Bhattacharya, which establishes Harnack's inequality by utilizing the fundamental solution in the case $p>n$.

\appendix

\section{Hitting probabilities for cylinder walk}

In this section we derive an estimate for the hitting probabilities for the 'linear part' i.e.\ the cylinder walk. The estimate is fairly standard, but we give a complete proof for the convenience of the reader.

There are two alternative arguments: The first one is to use a solution to the underlying linear PDE, see Remark~\ref{rem:alternative}. However, this approach is quite similar to that in Lemma~4.5 in \cite{manfredipr11}, and thus we have decided to use basic estimates in the stochastic analysis.

We start with properties of a standard random walk.   For consistency, we readily adopt the notation from the proof of Lemma~\ref{lemma:speed} below.
\begin{lemma}
\label{lemma:speed1} Let  $\varepsilon <1/4$ and let $\widetilde t_0,\widetilde t_1,\ldots$ denote positions in a  symmetric random walk on the real axis with steps $\pm \varepsilon $. The random walk is stopped upon reaching $\R\setminus (0,1)$,
and the associated stopping time is denoted by ${\widetilde \tau_g}$.
Then  for any $a>0$ we have
$$
\P\big( \widetilde t_{{\widetilde \tau_g}}\leq 0\big)\;\geq \; 1-(\widetilde t_0+\varepsilon)
 \quad {\rm and} \quad
\P\big({\widetilde \tau_g} \geq a\varepsilon^{-2}\big)\leq (\widetilde t_0+4\varepsilon)/a.
$$
\end{lemma}
\begin{proof}
Observe that $\widetilde t_j$ is martingale, and so is $\widetilde t^2_j-j\varepsilon^2$ since
$$((\widetilde t_j+\varepsilon)^2+(\widetilde t_j-\varepsilon )^2)/2 =\widetilde t_j^2+\varepsilon^2 .$$
Denote $p=\P (\widetilde t_{\widetilde \tau_g} \leq 0)$ and deduce from optional stopping that
$$
\widetilde t_0=\E \widetilde t_{{\widetilde \tau_g}} \geq p(-\varepsilon) +(1-p)\cdot 1,
$$
hence $p\geq(1-\widetilde t_0)(1+\varepsilon)^{-1}\geq 1-\widetilde t_0-\varepsilon.$ In a similar vain
$$
 {\widetilde t_0}^2=({\widetilde t_0}^2-0)=\E ((\widetilde t_{\widetilde \tau_g})^2- {\widetilde \tau_g}\varepsilon^2)\leq p\varepsilon^2+(1-p)(1+\varepsilon)^2-\varepsilon^2\E {\widetilde \tau_g},
$$
 yielding
$$\E {\widetilde \tau_g} \leq\varepsilon^{-2}\big((1+\varepsilon)^{2}-p-\widetilde t_0^2)\leq  \varepsilon^{-2}(\widetilde t_0+4\varepsilon),$$
where we also utilized the above estimate for $p$.
The second statement follows from this by an application of Chebychev's inequality.
\end{proof}

Recall next  a standard  large deviations bound for a sum of i.i.d.\ random variables. Consider  i.i.d.\ symmetric  $\R^n$-valued random variables  $Y_m$, $m=1,\ldots ,N$,   that are uniformly bounded:  $|Y_m|\leq b$  for all $m$. Then
\begin{equation}
\begin{split}
\label{eq:azumakolmogorov}
\P \big(\max_{1\leq m\leq N}|Y_1+\ldots +Y_m|\ge\lambda\big)\leq 4n\exp \big( -\frac{\lambda^2}{2Nb^2}\big),\qquad \lambda>0.
\end{split}
\end{equation}
With  factor 2 instead of 4 on the right hand side  this is just Hoeffding's (or Bernstein's or Azuma's) inequality (see \cite[Section 12.2]{GS01}), and the extra factor 2 comes from combining it with L\'evy-Kolmogorov's inequality
$\P\big(\max_{1\leq m\leq N}|Y_1+\ldots +Y_m|\ge\lambda\big)\leq 2\P\big(|Y_1+\ldots +Y_N|\ge\lambda\big)$, see \cite[Lemma 1 p. 397]{shiryaev}. Most of the stated  versions deal only with a one-dimensional situation, but  by considering  each component separately just adds the  extra factor $n$ to (\ref{eq:azumakolmogorov}).

We recall the definition of the cylinder walk:
 at a point $(x_j,t_j)\in B_{2r}(0)\times[0,2r]$, where $r>0$ is fixed and $B_{2r}(0)\subset \Rn$, we step with probability $\alpha/2$ to the point $(x_j,t_{j}-\eps)$, and with $\alpha/2$ to $(x_j,t_{j}+\eps)$. With probability $\beta$ we move to the point $(x_{j+1},t_j)$, where $x_{j+1}$ is randomly chosen from the ball $B_{\eps}(x_j)$.

Next we prove Lemma~\ref{lemma:speed}, which states that
when starting the cylinder-walk at the point $(0,t)$, then the probability that the walk does not escape the cylinder through its bottom is less than
$
C(p,n)(t+\varepsilon)/r.
$
\begin{proof}[Proof of Lemma~\ref{lemma:speed}] We first reduce by scaling to the case  $r=1/2.$ The cylinder walk can be equivalently constructed by combining  three  independent random constructions: The first item is  a  'horizontal' and symmetric random walk
where $\widetilde x_0=0$ and the point $\widetilde x_{j+1}$ is chosen according to a uniform distribution in $B_\eps(x_j)\subset\R^n$. The second is a random walk equivalent to the walk considered in Lemma \ref{lemma:speed1} with positions $\widetilde t_j$ and $\widetilde t_0=t$. The third is the increasing sequence $$U_j=\sum_{m=1}^jZ_m,$$ where the $Z_m$:s are independent Bernoulli variables with $Z_m\in \{0,1\}$ and $\P (Z_m=1)=\alpha .$ Then a copy of the cylinder walk is obtained by setting for $j\geq 0$
$$
t_j=\widetilde t_{U_j}, \quad  x_j=\widetilde x_{j-U_j}.
$$

 Apply first Hoeffding's  inequality with $Y_m=Z_m-\alpha,\,\lambda=j\alpha/2,\,b=1,\,N=j$ to get
\begin{equation}
\label{eq:level-set-for-bernoulli}
\begin{split}
\P (U_j\leq \alpha j/2)  \leq  \P (|U_j-j\alpha|  \geq  j\alpha/2)\leq 2\exp(-j\alpha^2/8).
\end{split}
\end{equation}
Set
\[
\begin{split}
A:=\{ U_ j>\alpha j/2  \quad {\rm for \; all }\; j\geq  1/\varepsilon \}.
\end{split}
\]
Then summing \eqref{eq:level-set-for-bernoulli} up for all indices  in the range  $j\geq1/\varepsilon$  yields
\begin{equation}
\begin{split}
\label{eq:bernoulli}
\P(A^c)&=\P( U_ j\le \alpha j/2  \quad {\rm for \; some }\; j\geq  1/\varepsilon)\\
&\le \P(\cup_{j\ge \frac{1}{\eps}}\{ U_ j\le \alpha j/2\})\\
&\le \sum_{j\ge 1/\eps}\P (U_j\leq \alpha j/2)\le  \sum_{j\ge 1/\eps}  2\exp(-j\alpha^2/8)\le O(\eps).
\end{split}
\end{equation}

Let $\tau_g$ (resp. $\widetilde \tau_g$) stand for the first moment  when $t_j$ (resp.
$\widetilde t_j$) steps out from  (0,1), and similarly  let  $\tau_b$ (resp. $\widetilde \tau_b$) stand for first time when $|x_j|$ (resp.
$|\widetilde x_j|$) reaches  $[1,\infty)$.  Here the subindex $g$ refers to the `good' stopping,  i.e. to the situation where the walk reaches the bottom or the top of the cylinder, and $b$  to the `bad' one where the walk reaches the vertical boundary.
Consider the  event
$$
B:=A\cap \{\widetilde \tau_b\geq 1/\varepsilon\}
$$
By applying  (\ref{eq:bernoulli}) and (\ref{eq:azumakolmogorov}) it follows that
\begin{equation}\label{eq:B-estimate}
\begin{split}
\P(B^c)&\leq \P(A^c)+\P(\widetilde \tau_b< 1/\varepsilon)\\
&\le  O(\varepsilon )+4n\exp(-1/(2\eps))\\
&\le O(\varepsilon),
\end{split}
\end{equation}
where in the second inequality we estimated with $Y_m=\tilde x_m$, $\lambda=1$, $b=\eps$, and $N=1/\eps$ that  $\P(\widetilde \tau_b< 1/\varepsilon)\le \P(\max_{1\le m\le N} \abs{\tilde x_1+\ldots +\tilde x_m}\ge \lambda)\le 4n \exp(-1/(2\eps)).$

It always holds that  $ \tau_b\geq \widetilde\tau_b,$  and so
in the set $\{\tau_b\leq\tau_g\}\cap  B$ one has  $\tau_g\ge \tau_b\ge \widetilde \tau_b\geq 1/\varepsilon$.
Since  $\tau_g\ge 1/\eps$ in $\{\tau_b\leq\tau_g\}\cap  B \subset A$, we may apply the definition of $A$ to deduce
$$\widetilde\tau_g\leq  \tau_g\leq (2/\alpha)\widetilde\tau_g,$$
i.e. $\{\tau_b\leq\tau_g\}\cap  B \subset \{\widetilde\tau_g\le  (2/\alpha)\widetilde\tau_g\}.$
 Thus by Lemma \ref{lemma:speed1} and the independence of $\widetilde \tau_g$ and $\widetilde \tau_b$ we can  estimate
\[
\begin{split}
&\P (\{\tau_b\leq\tau_g\}\cap  B)
\leq\;\P (\widetilde \tau_g> \varepsilon^{-2})+\P (\widetilde\tau_b\leq(\alpha/2)\widetilde\tau_g \;\; {\rm and}\;\;  \widetilde \tau_g\leq \varepsilon^{-2})\\
\leq &\; O(t)+\sum_{\ell=1}^{\infty}
\P \big( \varepsilon^{-2}/(\ell+1)\leq \widetilde\tau_g\leq \varepsilon^{-2}/\ell\;\; {\rm and}\;\; \widetilde \tau_b\leq (2/\alpha) \varepsilon^{-2}/(\ell+1) \;\big)\\
\leq &\; O(t)+\sum_{\ell=1}^{\infty}
\P \big( \widetilde\tau_g\geq\varepsilon^{-2}/(\ell+1)\big)\; \P\big(\widetilde \tau_b\leq (2/\alpha) \varepsilon^{-2}/(\ell+1) \;\big).
\end{split}
\]
Then by using Lemma~\ref{lemma:speed1} again and estimating $\P\big(\widetilde \tau_b\leq (2/\alpha) \varepsilon^{-2}/(\ell+1) \;\big)$ by using \eqref{eq:azumakolmogorov} similarly as before, we get
\[
\begin{split}
\P (\{\tau_b\leq\tau_g\}\cap  B)&\le O(t)+ \sum_{\ell=1}^{\infty}
\big((t+4\eps)(\ell+1)\big)\big(4n\exp (-(\ell+1)\alpha/4)\big)\\
&\le O(t+\eps).
\end{split}
\]
By combining this  with (\ref{eq:B-estimate}) and Lemma \ref{lemma:speed1} we get  the desired upper bound
$$
\P (\tau_b \leq\tau_g  \;{\rm or}\; t_{\tau_g}\ge1)\leq \P (\{\tau_b\leq\tau_g\}\cap  B)+ P(B^c)+(t+\varepsilon)=O(t+\varepsilon).\qedhere
$$
\end{proof}

\begin{remark}
\label{rem:alternative}
{\rm Let us   outline an alternative proof for Lemma \ref{lemma:speed}, again assuming $r=1/2.$ Write $U:= B_1(0)\times (0,1).$ Find a solution $u\leq 1$  to a linear PDE $(p-2)u_{tt}+u_{x_1x_1}+\ldots+u_{x_n x_n}=0$ in a larger domain so that $$u(x,-\eps)=1\quad\trm{for}\quad x\in B_{1/2}(0),$$  and such that $$u\leq 0\quad\trm{on}\quad\partial U \setminus (B_1(0)\times \{0\}).$$ An explicit solution to this problem can be obtained by scaling from a harmonic function. Consider the sequence of random variables $u(x_j,t_j)$, $j=1,2,\ldots ,$ where
$(x_j,t_j)_{j\in\N}$ are the positions in the cylinder walk. Now $u$ satisfies
\begin{equation}\label{eq:almostdpp}
\begin{split}
u(x,t)=\beta \kint_{B_\eps(x)} u \ud y&+\frac{\alpha}{2}\Big(u(x,t+\eps)+u(x,t-\eps)\Big) + O(\varepsilon^3).
\end{split}
\end{equation}
This makes $M_j:=u(x_j,t_j)+cj\varepsilon^3$ a submartingale.
Hence by optional stopping, using the stopping time $\tau'$ that corresponds to  exit from $U$, we  deduce that
$$\mathbb E[u(x_{\tau'},t_{\tau'})+c\tau' \eps^3]\ge M_0=u(0,t).$$ We recall that $\mathbb E[u(x_{\tau'},t_{\tau'})]$ gives  a lower bound for the desired probability of exiting through the bottom, up to certain error. This error  arises from not stopping exactly on the boundary of $U$, but it  is  clearly of order $O(\varepsilon)$ by the smoothness of $u$.  Now $\E \tau'\leq C\varepsilon^{-2},$ as a simple variant of Lemma \ref{lemma:speed1} shows, applied just  on  the $t$-component of the walk.  Finally the result follows by combining these observations with the  inequality $u(0,t)\geq 1-c't$.
}
\end{remark}

\def\cprime{$'$} \def\cprime{$'$}

\end{document}